\documentclass[preprint,fleqn]{elsarticle}

\usepackage{amsmath}
\usepackage{amssymb}
\usepackage{amsfonts}
\usepackage{enumerate}
\usepackage[all,cmtip]{xy}
\usepackage{bm}
\usepackage{bbm} 
\usepackage{verbatim} 

\def\C 	{\ensuremath{\mathbb{C}}}

\def\CP	{\ensuremath{\mathbb{CP}}}

\def\bbG 	{\ensuremath{\mathbb{G}}}

\def\P 	{\ensuremath{\mathbb{P}}}
\def\Q 	{\ensuremath{\mathbb{Q}}}
\def\R 	{\ensuremath{\mathbb{R}}}

\def\bbT 	{\ensuremath{\mathbb{T}}}

\def\Z 	{\ensuremath{\mathbb{Z}}}

\def\calB	{\ensuremath{\mathcal{B}}}

\def\calH	{\ensuremath{\mathcal{H}}}

\def\calO   {\ensuremath{\mathcal{O}}}

\def\d 		{\ensuremath{\mathrm{d}}}

\def\dist 	{\ensuremath{\mathrm{dist}}}

\def\GL    	{\ensuremath{\mathrm{GL}}}

\def\loc		{\ensuremath{\mathrm{loc}}}

\def\op		{\ensuremath{\mathrm{op}}}

\def\pr		{\ensuremath{\mathrm{pr}}}

\def\Rm		{\ensuremath{\mathrm{Rm}}}

\def\SL    	{\ensuremath{\mathrm{SL}}}

\def\Tr    	{\ensuremath{\mathrm{Tr}}}

\def\BC		{\ensuremath{\mathbf{BC}}}

\newcommand{\abs}[1]{\left|#1\right|}
\newcommand{\norm}[1]{\left\| #1 \right\|}

\newcommand{\ip}[1]{\left\langle #1 \right\rangle}
\newcommand{\set}[1]{\ensuremath{\left\{ #1 \right\}}}

\def\ddbar  {\ensuremath{\partial\partialbar}}

\def\dz		{\ensuremath{\d z}}
\def\dzbar	{\ensuremath{\d\zbar}}
\def\inv		{\ensuremath{^{-1}}}

\def\minus	{\ensuremath{\smallsetminus}}
\def\too		{\ensuremath{\rightarrow}}
\def\smooth	{\ensuremath{C^\infty}}

\def\I		{\ensuremath{\sqrt{-1}}}

\def\sl    	{\ensuremath{\mathfrak{sl}}}

\def\bdef		{\begin{definition}}
\def\enddef		{\end{definition}}
\def\bthm		{\begin{theorem}}
\def\ethm		{\end{theorem}}
\def\bfig		{\begin{figure}}
\def\efig		{\end{figure}}

\def\be 			{\begin{equation}}
\def\ee 			{\end{equation}}

\newcommand{\cl}[1]{\ensuremath{\overline{#1}}}

\def\hbar 	{\ensuremath{\bar{h}}}

\def\jbar 	{\ensuremath{\bar{j}}}

\def\qbar 	{\ensuremath{\bar{q}}}

\def\zbar 	{\ensuremath{\bar{z}}}

\def\ellbar	{\ensuremath{\overline{\ell}}}

\def\partialbar{\ensuremath{\bar{\partial}}}

\usepackage{amsthm}

\newtheorem{Thm}{Theorem}[section]
\newtheorem{cor}{Corollary}
\newtheorem{lem}{Lemma}

\theoremstyle{definition}
\newtheorem{definition}{Definition}
\newtheorem{remark}{Remark}
\newtheorem{rems}{Remarks}

\usepackage{pifont}
\usepackage{natbib}
\usepackage{geometry}
\usepackage{fleqn}
\usepackage{graphicx}
\usepackage{txfonts}
\usepackage{hyperref}
\usepackage{lineno}

\newcommand\lognorm[1] {\ensuremath{\log\frac{\norm{\sigma\cdot{#1}}^2}{\norm{#1}^2}}}

\def\Szekelyhidi{Sz{\'e}kelyhidi}
\def\d{\ensuremath{\mathrm{d}}}

\journal{arXiv.org}

\begin{document}

\title{Discriminants and Higher K-energies on Polarized K\"ahler Manifolds}

\author[wisc]{Quinton Westrich}
\ead{westrich@math.wisc.edu}
\address[wisc]{Department of Mathematics, University of Wisconsin, 480 Lincoln Dr., Madison, WI 53706-1388}

\begin{frontmatter}

\begin{abstract}
  Given a compact polarized K\"ahler manifold $X\hookrightarrow\CP^N$, the space of Bergman metrics on $X$, parameterized by $\SL(N+1,\C)$, corresponds to a dense set in the space of K\"ahler potentials in the K\"ahler class as $N\too\infty$. Critical points of the $k$th K-energy functional, which is defined on the K\"ahler class, correspond to metrics with harmonic $k$th Chern form.
In this paper it is shown that the higher K-energy functionals, when restricted to the Bergman metrics, are expressible as the energies of  certain pairs of vectors (tensors products of discriminants). Consequentially, we obtain results on the asymptotic behavior of these functionals along 1-parameter subgroups and their boundedness properties.
\end{abstract}


\end{frontmatter}


\section{Introduction}


A major body of research in K\"ahler geometry has been guided by the Tian-Yau-Donaldson Problem, which asks for necessary and sufficient conditions for the existence of a canonical metric (e.g., extremal, cscK, K\"ahler-Einstein) in a given Hodge class. This problem has been solved in the Fano case by Chen-Donaldson-Sun \cite{CDSI,CDSII,CDSIII} and Tian \cite{tian15} in 2012. Alternatively, with the establishment of the partial $C^0$ estimate by \Szekelyhidi\ in 2013 \cite{szekelyhidi15}, S.\ Paul's 2012 paper \cite{paul12} showed that for a Fano manifold $X$ with finite automorphism group, $(X,-K_X)$ is asymptotically K-stable (in Paul's sense of \emph{pairs}) if and only if it admits a K\"ahler-Einstein metric.

A key step in \cite{paul12} was the algebraic reformulation of the Mabuchi K-energy in terms of the classical algebro-geometric \emph{discriminants}, i.e., defining polynomials of hypersurface dual varieties.
The analytically defined \emph{K-energy map} was defined by Mabuchi in 1986 (\cite{mabuchi86}) in order to detect K\"ahler-Einstein metrics in a given K\"ahler class. The Mabuchi K-energy is an integrated form of the \emph{Futaki invariant}, which vanishes if the K\"ahler class admits a cscK metric. Accordingly, a cscK metric is an extremum of the Mabuchi K-energy. Paul's reformulation of the Mabuchi K-energy in algebraic terms thus allowed for a GIT-style stability criterion to replace the extremal condition for K\"ahler-Einstein metrics. It is hoped that the stability criterion will be easier to check in explicit examples.

In the 1986 paper \cite{BM86}, Bando and Mabuchi defined a broader class of functionals $M_k$, $k=1,2,\ldots$, that generalized the Mabuchi K-energy, the case $k=1$. These \emph{higher K-energy functionals} integrate a corresponding class of higher Futaki invariants (\cite{bando06}\footnote{written in 1983, but not published until 2006}). Moreover, a K\"ahler metric with harmonic $k$th Chern form gives an extremum of the $k$th K-energy. However, for $k>1$ very little is known about the behavior of these functionals. It is unknown whether these functionals are bounded above or below, or whether they enjoy any sort of convexivity properties, in analogy with the case $k=1$.


We now outline the key results of this paper. As a preliminary result, we obtain the following
\begin{Thm} \label{thm:mainthm}
  Let $X^n\hookrightarrow\CP^N$ be a smooth, nonlinear, irreducible, subvariety embedded by a complete linear system and $k\leq n$ be a positive integer. 
  If $k\geq 3$, assume further that at least one of the Chern classes $c_j(J_1(\calO_X(1)))\neq 0$ for $k\leq j\leq n$, where $J_1(\calO_X(1))$ is the bundle of 1-jets of the hyperplane bundle.
  
  Then there exist $\SL(N+1,\C)$-modules $V_k$ and $W_k$ equipped with Hermitian norms, and nonzero vectors $v_k\in V_k$ and $w_k\in W_k$ such that the $k$th K-energy restricted to the Bergman metrics is given by
  \be \label{eq:maineq}
    M_k(\sigma) = \lognorm{v_k}-\lognorm{w_k},
  \ee
  for $\sigma\in\SL(N+1,\C)$.
\end{Thm} 

The $\SL(N+1,\C)$-modules $V_k$ and $W_k$ and the vectors $v_k$ and $w_k$ are given explicitly in Section \ref{sec:higherKen} in terms of Chow forms and $X$-discriminants.
The construction of the norms, due to Tian (\cite{tian94}), is given in Section \ref{sec:logpolygrowth}. The assumption on the Chern classes of the jet bundle is there to ensure the relevant $X$-discriminants exist.

With Theorem \ref{thm:mainthm} in tow, we obtain two results on the global behavior of $M_k$ on the space of Bergman metrics $\calB$.
\begin{cor}[Asymptotics of $M_k$] \label{cor:asymptotics}
  Let $\lambda:\C^*\to\SL(N+1,\C)$ be a 1-parameter subgroup. Then there exist asymptotic expansions of the higher K-energies as $\abs{t}\to 0:$
  \be
    M_k(\lambda(t)) = A_k(\lambda) \log\abs{t}^2 + B_k(\lambda),
  \ee
  where $A_k(\lambda)\in\Z$ and $B_k(\lambda)$ is $O(1)$.
\end{cor}

\begin{cor}[Boundedness of $M_k$] \label{cor:boundedness}
  The following are equivalent:
  \begin{enumerate}
    \item $M_k$ is bounded below on $\calB$ 
    \item $M_k$ is bounded along all 1-parameter subgroups $\lambda:\C^*\to\SL(N+1,\C)$
    \item $M_k$ is bounded below on all algebraic tori in $\SL(N+1,\C)$
    \item the pair $(v_k,w_k)$ is K-semistable.
  \end{enumerate}
\end{cor}


This paper is organized as follows. In Section \ref{sec:notation} we establish notation and recall the definitions of the higher K-energies and discriminants. We set the notation for the embedding of $X$ into $\CP^N$ and the related Bergman metrics and then define the higher K-energies. Finally, we define discriminants and recall certain facts from the literature that relate to our work.
In Section \ref{sec:logpolygrowth} we show a technical lemma: the higher K-energies have $\log$-polynomial growth on \mbox{$\SL(N+1,\C)$.} This fact enables us to employ Tian's ``$\ddbar$'' technique from \cite{tian94}, needed for Lemma \ref{lem:lognorm}.
In Section \ref{sec:discrdegrees} we obtain a formula discriminant degrees $\deg\left(\Delta_X^{(n-k)}\right)$ in terms of the integrals $\mu_k$ defined in Section \ref{subsec:higherenergies}. This requires computing a top Chern class $c_{2n-k}$ on the bundle of 1-jets on the hyperplane bundle $J_1(\calO_{s(X\times\CP^{n-k})}(1))$. In Section \ref{sec:higherKen}, this Chern class is then used to obtain the formula (\ref{eq:maineq}), relating higher K-energies to discriminants.

The author would like to thank his PhD advisor, Sean Paul, for encouraging him to study these functionals, and for helpful advice along the way. Also, he would like to acknowledge Joel Robbin, Jeff Viaclovsky, and Bing Wang for their roles in his education. Finally, he would like to thank the mathematics department at UW-Madison for its stimulating environment. This work will provide part of his dissertation.

\section{Background and Notation} \label{sec:notation}

\subsection*{Polarized K\"ahler Manifolds} \label{subsec:embedding}

A K\"ahler manifold $(X,\omega)$ is said to be \emph{polarized} by a Hermitian holomorphic line bundle $(L,h)$ over $X$ if $\omega$ is the curvature of $h$, i.e., if $\omega=-\I\ddbar\log \abs{s}^2_h$ for any local non-vanishing holomorphic section $s$. By Kodaira's embedding theorem any compact polarized K\"ahler manifold holomorphically embeds into $\P(H^0(X,L^{\otimes m}))\cong\CP^N$, for $m$ sufficiently large and $N:=\dim_{\C} H^0(X,L^{\otimes m})-1$. 

Conversely, any projective embedding $X\hookrightarrow \CP^N$ polarizes $X$.
Specifically, given a smooth, compact, projectively embedded K\"ahler manifold $\iota:X\hookrightarrow \CP^N$ of complex dimension $n$ with K\"ahler form $\omega$, we have that $X$ is polarized by the (positive, holomorphic) line bundle $\calO_{\CP^N}(1)|_{\iota(X)}$. The standard metric $h_{\C^{N+1}}$ on $\C^{N+1}$ restricts to a Hermitian metric $h$ (unique up to rescaling) on $\calO_{\CP^N}(-1)|_{\iota(X)}\subset\CP^N\times\C^{N+1}$ such that $\omega$ is the curvature $F$ of the Chern connection $\nabla^{h^\vee}$ for $h^\vee$, the metric on $\calO_{\CP^N}(1)|_{\iota(X)}:=\calO_{\CP^N}(-1)|_{\iota(X)}^\vee$ induced by $h$; thus
$ \omega = \iota^*F(\nabla^{h^\vee}).$
A local nonvanishing holomorphic section $s$ of $\calO_{\CP^N}(1)|_{\iota(X)}$ gives a local K\"ahler potential via
\be
  \omega =_\loc -\I \iota^* \ddbar \log \abs{s}^2_{h^\vee}
  = \I \iota^* \ddbar \log \abs{s}^2_{h},
\ee
where $\abs{s}^2_{h^\vee}$ is the (pointwise) square norm of $s$ with respect to $h^\vee$ and similarly for $\abs{s}^2_h$.

\subsection*{Bergman Metrics} \label{subsec:bergman}

Given a basis $\set{s_0,\ldots,s_N}$ of sections of $H^0\!\left(X,\iota^*\calO_{\CP^N}(1)|_{\iota(X)}\right)$, the embedding $\iota$ can be written explicitly for an open set $U_i\subset X$ on which some $s_i$ is nonvanishing as
$\iota|_{U_i} :U_i\hookrightarrow \CP^N $
given by
\begin{align} 
\iota|_{U_i} &:p\mapsto \left[\frac{s_0(p)}{s_i(p)}:\cdots:\frac{s_N(p)}{s_i(p)} \right].
\end{align}
Let $z=(z^1,\ldots,z^n):U_0\to\C^n$ be local coordinates on $U_0$ and set 
 $T:\C^n\to\C^{N+1} $
 \begin{align} 
 T_i(z(p))&:=\frac{s_i(p)}{s_0(p)} ,
\end{align} 
$i=0,\ldots,N$,
so that 
$ \iota|_{U_0}(p) = [1:T_1(z(p)):\cdots:T_N(z(p))]. $
The K\"ahler metric $\omega = \frac{\I}{2\pi}\sum g_{i\jbar}\,\d z^i\wedge\d \zbar^j$ on $X$ is the pullback under $\iota$ of the Fubini-Study metric on $\CP^N$ restricted to $\iota(X)$. Let $\abs{\;\cdot\:}$ and $\ip{\,\cdot\,,\,\cdot\,}$ denote the norm and inner product on $\C^{N+1}$, respectively, and put $\partial_i:=\partial_{z^i}$, etc.
Since $T$ is holomorphic, we have
\begin{align}
  g_{i\jbar}(z) 
  &= \partial_i \partial_{\jbar} \log \abs{T(z)}^2 
  = \partial_i \left(\frac{ \ip{T(z),\partial_jT(z) } }{\abs{T(z)}^2} \right) 
  = \frac{\ip{\partial_{i}T(z),\partial_jT(z)}}{\abs{T(z)}^2} - \frac{\ip{T(z),\partial_jT(z)}\ip{\partial_iT(z),T(z)}}{\abs{T(z)}^4} . 
\end{align}

The standard action of $\SL(N+1,\C)$ on $\C^{N+1}$ induces an action of $\SL(N+1,\C)$ on $X$ so that for $\sigma\in\SL(N+1,\C)$ and $p\in X$,
$$
  \sigma\cdot p := \left[\sigma\cdot \left(\sum s_i(p)e_i\right)\right],
$$
where the $e_i$ are the standard basis vectors in $\C^{N+1}$. The metric on $\sigma\cdot X$ is given locally by
\begin{align}
  \omega_\sigma(z)
  := \sigma^*\omega (z)
  &=_\loc \I\iota^*\ddbar\log\abs{\sigma\cdot T(z)}^2 .
\end{align}
It follows that $\omega_\sigma = \omega+\I\ddbar\varphi_\sigma$, where
\begin{align}
  \varphi_\sigma(z) &= \log \frac{\abs{\sigma\cdot T(z)}^2}{\abs{T(z)}^2}
\end{align}
on $U_0\subset X$. Writing $\omega_\sigma = \frac{\I}{2\pi}\sum h_{i\jbar}\,\d z^i\wedge\d \zbar^j$ on $U_0$, we have
\begin{align}
    h_{i\jbar}(z)
    &= \frac{\ip{\sigma(\partial_iT(z)),\sigma(\partial_jT(z))}}{\abs{\sigma T(z)}^2} - \frac{\ip{\sigma(\partial_iT(z)),\sigma T(z)}\ip{\sigma T(z),\sigma(\partial_jT(z))}}{\abs{\sigma T(z)}^4}.
  \end{align}
  
  For a general compact K\"ahler manifold the space of K\"ahler metrics in the cohomology class $[\omega]$ are parameterized by real $C^\infty$ plurisubharmonic functions up to the the addition of a constant via $\omega_\varphi=\omega+\sqrt{-1}\ddbar\varphi$, where $\varphi\in\smooth(M,\R)$. We denote the function space
$$ \calH_\omega := \set{\varphi\in\smooth(M,\R)\,:\, \omega_\varphi=\omega+\sqrt{-1}\ddbar\varphi>0}, $$
where here $>0$ means ``is positive definite.'' The key result due to Tian, Ruan, Zelditch, and Catlin shows that the Bergman metrics are dense in the $C^\infty$ topology on $\calH_\omega$, in the sense that for any $\varphi\in\calH_\omega$ there exists a sequence $\frac{1}{k}\rho_k:= \frac{1}{k} \log(\sum_j\abs{s_j}^2_{h^k})$ converging to $\varphi$ as $k\too\infty$ in the $C^\infty$ topology, where $\set{s_0,\ldots,s_{N_{k}}}$ is a basis of $H^0(X,L^{\otimes k})$.

\subsection*{Higher K-Energies}
\label{subsec:higherenergies}
Let $(X,\omega)$ be a compact K\"ahler manifold, $\varphi\in\calH_{\omega}$, and $k\in\set{1,\ldots,n}$.
Given a smooth path $\Phi:[0,1]\to\calH_{\omega_0}$, $\varphi_t:=\Phi(t)$, from $\varphi_0:=0$ to $\varphi_1:=\varphi$,
the \emph{$k$th K-energy functional} $M_k:\calH_\omega\to\R$ is defined to be
\be \label{eq:genKenergy}
  M_k(\varphi):= -(n+1)(n-k+1)V\int_0^1 \set{ \int_X \dot\varphi_t \left[c_k(\omega_t)-\mu_k\omega_t^k\right]\wedge \omega_t^{n-k} } \d t,
\ee
where 
\begin{align} 
 \omega_t&:=\omega_0+\I\ddbar\varphi_t & \mu_k &:= \frac{1}{V} \int_X c_k(\omega_0)\wedge \omega_0^{n-k} \label{eq:muk} &
  V &:= \int _X \omega_0^n.
\end{align}
Note that $\mu_k$ and $V$ are constants on $\calH_\omega$. The factor $-(n+1)(n-k+1)V$ coincides with the normalization for the Mabuchi K-energy in \cite{paul12}. Its presence simplifies the formula for $M_k$ in terms of discriminants and ensures that $A_k(\lambda)$ in Corollary \ref{cor:asymptotics} is in $\Z$ instead of just $\Q$.
Bando and Mabuchi showed that these functionals are independent of the chosen path $\omega_t$ in $\calH_{\omega}$.

In the case $k=1$, $M_1$ is the Mabuchi energy whose extrema are cscK metrics. In the general case, the extrema of $M_k$ are metrics whose $k$th Chern form is harmonic.

\subsection*{Discriminants}
\label{subsec:discriminants}

In this subsection we provide the necessary background material on discriminants and projective duality. While many of the definitions and results are classical, we include them here since they may not be familiar to many K\"ahler geometers. A good reference for the material in this section is the excellent book by Gelfand, Kaparanov, and Zelevinsky \cite{GKZ08}. See also \cite{tevelev05} for a more compact treatment.

Let $X^n\hookrightarrow\CP^N$ be an projectively embedded K\"ahler manifold. For each $p\in X$, denote by $\bbT_pX$ the \emph{embedded tangent space}  to $X$ at $p$. This is an $n$-dimensional linear subspace of $\CP^N$. The set of all hyperplanes in $\CP^N$ is the \emph{dual} of $\CP^N$, denoted $(\CP^N)^\vee$.

\bdef
Let $X\hookrightarrow\CP^N$ be a projectively embedded K\"ahler manifold. Assume further that $X$ is a nonlinear, linearly normal subvariety. Then the \emph{dual variety} $X^\vee$ of $X$ is the variety of tangent hyperplanes to $X$:
\be
  X^\vee := \cl{\set{H\in(\CP^N)^\vee\,|\, \exists p\in X:\bbT_pX \subseteq H}} ,
\ee
the closure taken in the Zariski sense. 
\enddef

The \emph{linear normality} condition is added to avoid trivialities, and is not restrictive. It ensures that $X$ is nondegenerate and not equal to a nontrivial projection. Here, \emph{nondegeneracy} means that $X\subset\CP^N$ is not contained in any hyperplane. The essential content is that our embedding is optimal: smaller $N$ values preclude an embedding. 

\begin{quote}
\emph{In the remainder of this section it is assumed that $X^n\hookrightarrow\CP^N$ is a smooth, nonlinear, irreducible, linearly normal, degree $d$ projective variety with $n<N$.}
\end{quote}

Most dual varieties are hypersurfaces in $(\CP^N)^\vee$. The deviance of the codimension of $X^\vee\subset(\CP^N)^\vee$ from $1$ is measured by the \emph{dual defect} 
\be
  \delta(X) := (N-1)-\dim(X^\vee) \geq 0.
\ee
An upper bound for
$X^n\hookrightarrow \CP^N$ with $n\geq 2$ is
$$ \delta(X) \leq n-2. $$
We have a formula for the dual defect in terms of the Chern classes of the bundle of $1$-jets on the hyperplane bundle
$$ \delta(X) = \min\set{k\in\Z\,|\,c_{n-k}(J_1(\calO_X(1)))\neq 0}. $$

\bdef
If $\delta(X)=0$ so that $X^\vee\subset(\CP^N)^\vee$ is a hypersurface, the defining polynomial $\Delta_X$ (unique up to scaling) is called the \emph{$X$-discriminant}:
\be
  (\Delta_X)\inv(0) := X^\vee  \subset (\CP^N)^\vee.
\ee
\enddef
We usually just speak of \emph{the} discriminant when $X$ is understood.

We can say more about the dual defect if we follow Cayley and look at Segre embeddings. 

\bdef
In general, we consider the Segre embedding
$$X\times\CP^k \hookrightarrow \P(\C^{N+1}\otimes\C^{k+1}).$$ 
If $\delta(X\times\CP^k)=0$, the \emph{$X$-hyperdiscriminant} of \emph{format}\footnote{As the notation suggests, there is a multiindex fomulation of the $X$-hyperdiscriminant. We omit this as it is unnecessary for our purposes.} $(k)$ is the irreducible defining polynomial of the hypersurface $(X\times\CP^k)^\vee$:
\be
  (\Delta^{(k)}_X)\inv(0):=(X\times\CP^k)^\vee \subset \P(\C^{N+1}\otimes\C^{k+1})^\vee .
\ee
\enddef

\begin{lem}
The $X$-hyperdiscriminant $\Delta^{(k)}_X$ exists if and only if 
\be \delta(X)\leq k\leq n.\ee
\end{lem}
In particular, since $\delta(X)\leq n-2$ whenever $n\geq 2$,
we have in that case that 
$\Delta_{X}^{(n-2)},\Delta_{X}^{(n-1)},\Delta_{X}^{(n)}$ always exist.


There is a nice relationship between discriminants and resultants. Recall that 
the \emph{Cayley-Chow form} of $X$, or \emph{$X$-resultant}, is the defining polynomial $R_X$ (unique up to scaling) of the divisor
\be
 (R_X)\inv(0) := \set{L\in\bbG(N-n,\CP^N) \,|\, L\cap X\neq\varnothing} ,
\ee
where $\bbG(k,\CP^N)$ denotes the Grassmannian variety of $k$-planes in $\CP^N$.
We note that $R_X$ is irreducible since $X$ is irreducible, and in Pl\"ucker coordinates on $\bbG(N-n,\CP^N)$, $\deg(R_X)=\deg(X)$.
The \emph{Cayley trick} relates $X$-hyperdiscriminants and $X$-resultants by
\begin{align}
  \Delta_{X}^{(\delta(X))} &= R_{X^\vee} \\
  \Delta_{X}^{(n)} &= R_{X} .
\end{align}
We think of intermediate hyperdiscriminants $\Delta_{X}^{(k)}$ with $\delta(X)<k<n$ as interpolating between $R_{X^\vee}$ and $R_X$.

\section{Log-Polynomial Growth of K-Energies on the Space of Bergman Metrics} \label{sec:logpolygrowth}

The purpose of this section is to generalize the Main Lemma in \cite{paul12}, stated below, to the higher K-energies. Recall that the \emph{Donaldson functional} of a $\GL(n,\C)$-invariant polynomial $\Phi$ of degree $n+1$ on a vector bundle $E$ is
\be
  D_E(\Phi;H_0,H_1) := \int_X \BC(E,\Phi;H_0,H_1),
\ee
where $\BC(E,\Phi;H_0,H_1)$ is the \emph{Bott-Chern form} of $\Phi$ on the vector bundle $E$ between the Hermitian metrics $H_0$ and $H_1$ on $E$. The Bott-Chern form \emph{transgresses} between $\Phi(F_0)$ and $\Phi(F_1)$, i.e.
\be
  \I\ddbar \BC(E,\Phi;H_0,H_1)
  = \Phi(F_1)-\Phi(F_0).
\ee

\begin{lem}[\cite{paul12}] \label{lem:paul}
  Let $X\hookrightarrow \CP^N$ be a smooth, linearly normal $n$-dimensional subvariety. Assume that $X^\vee$, the dual of $X$, is a hypersurface with defining polynomial $\Delta_X$ of degree $d^\vee$ and that $D_{J_1(\calO_X(1))^\vee}(c_{n+1};H(\sigma),H(e))$ has log-polynomial growth in $\sigma$, where $\sigma\in\SL(N+1,\C)$. Then there is a continuous norm $\norm{\,\cdot\,}$ on the vector space of degree-$d^\vee$ polynomials on $(\C^{N+1})^\vee$ such that for all $\sigma \in \SL(N+1,\C)$, we have
  \be
    (-1)^{n+1} D_{J_1(\calO_X(1))^\vee}(c_{n+1};H(\sigma),H(e))
    = 
    \lognorm{\Delta_X}
  \ee
  where $e$ denotes the identity of $\SL(N+1,\C)$. 
\end{lem}

Here we recall the construction of the continuous norm on $\calO_B(-1)$, where $B$ is the projective space
\be
  B:= \P(H^0((\CP^N)^\vee,\calO(d^\vee)))
\ee
and $d^\vee:=\deg(X^\vee)$. The discriminant $\Delta_X\in B$ and, given a linear functional $a_0z_0+\cdots+a_Nz_N$ on $(\CP^N)^\vee$, we can write
\be
  \Delta_X(a_0z_0+\cdots+a_Nz_N) = \sum_{i_0+\cdots+i_N=d^\vee} c_{i_0,\ldots,i_N} a_0^{i_0} \cdots a_N^{i_N}.
\ee
In these coordinates, we define a norm on $\calO_B(-1)$ by
\be
  \norm{\Delta_X}^2_{FS}
  := \sum_{i_0+\cdots+i_N=d^\vee} \frac{\abs{c_{i_0,\ldots,i_N}}^2}{i_0! \cdots i_N!}.
\ee
Define a new norm conformal to $\norm{\,\cdot\,}_{FS}$ by
\be
  \norm{\,\cdot\,} := e^\theta \norm{\,\cdot\,}_{FS},
\ee
where $\theta$, defined below, is a continuous function on $B$. Note that $\theta$ is bounded since $B$ is compact. 

To define $\theta$, first recall that the \emph{universal hypersurface associated to $B$} is the kernel of the evaluation map
\be
  \Sigma := \set{([F],[a_0z_0+\cdots+a_Nz_N])\in B\times(\CP^N)^\vee\,:\, F(a_0,\ldots,a_N)=0}.
\ee
Now define $u$ to be the $(1,1)$-current on $B$ given by 
\be
  \int_B u\wedge \psi = \int_\Sigma (\pr_2)^*(\omega_{(\CP^N)^\vee})\wedge(\pr_1)^*(\psi)
\ee
for all smooth $(b-1,b-1)$-forms $\psi$ on $B$, where $b=\dim_\C B$, $\omega_{(\CP^N)^\vee}$ is the Fubini-Study K\"ahler form on $(\CP^N)^\vee$, and $\pr_1$ and $\pr_2$ are the projection maps on $B\times(\CP^N)^\vee$. Tian (\cite{tian97}) showed that $[u]=[\omega_B]$ in cohomology, where $\omega_B$ is the Fubini-Study form on $B$, and there exists a continuous function $\theta$ on $B$ such that
\be
  u=\omega_B + \I\ddbar\theta,
\ee
in the sense of currents.

We explain the ``log-polynomial growth'' mentioned in the lemma.
Denote
\be
  D(\sigma) := D_{J_1(\calO_X(1))^\vee}(c_{n+1};H(\sigma),H(e)).
\ee
In the proof of this lemma it was shown that (Prop.\ 4.2) the quantity
\be
  (-1)^{n+1}D(\sigma) - \lognorm{\Delta_X}
\ee
is a pluriharmonic function on $G$, i.e.
\be
  \ddbar\left((-1)^{n+1} D(\sigma) - \lognorm{\Delta_X} \right) = 0.
\ee
We would like to drop the $\ddbar$ from this formula.
To this end, note that pluriharmonicity implies that there is a holomorphic function $F$ on $G$ such that
\be
  (-1)^{n+1}D(\sigma) - \lognorm{\Delta_X} = \log\abs{F(\sigma)}^2.
\ee
Following \cite{tian97} pp.33--34, consider $G$ as a quasi-affine subvariety of $\CP^{(N+1)^2}$. More precisely, given homogeneous coordinates $z_{ij}$ for $0\leq i,j\leq N$, define $W$ to be the affine variety
$$ W = \set{[z_{00}:z_{01}:\cdots:z_{N,N}:w]\,:\, \det(z_{ij}) = w^{N+1} } \subset \CP^{(N+1)^2}. $$
Then $G=W\cap\set{w\neq 0}$.
We have that $F$ extends to $W$ as a meromorphic function provided $F$ \emph{grows polynomially} near $W\minus G$, i.e.
there are constants $\ell>0$ and $C>0$ such that 
$$ F(\sigma) \leq C \cdot \dist(\sigma,W\minus G)^\ell, $$
where the distance is measured using the Fubini-Study metric on $\CP^{(N+1)^2}$.
All the poles of $F$ must live in $W\minus G$. But $W\minus G$ is irreducible and $W$ is normal, so all the zeroes of $F$ must live in $W\minus G$. Therefore $F$ is constant and, since $\log\abs{F(e)}^2=0$ for $e\in G$ the identity,
\be
  (-1)^{n+1}D(\sigma) = \lognorm{\Delta_X}.
\ee

While the polynomial growth of $F$ was given for $D(\sigma)$ corresponding to $M_1(\sigma)$, we must establish the log-polynomial growth for the higher K-energies. This is done in the next two lemmas.

\begin{lem}
  Given a compact polarized K\"ahler manifold $\iota:X\hookrightarrow \CP^n$, let $\omega_\sigma$ denote the Bergman metric induced by $\sigma\in\SL(N+1,\C)$.
Then there exist constants $C_1,C_2,C_3,C_4$ such that
  \begin{align}
    \label{eq:hest}
    \norm{\omega_\sigma}_{\omega_e} 
    &\leq C_1(\norm{\iota}_{C^1(X)}) \\
    \label{eq:hinvset}
    \norm{\omega_\sigma^\#}_{\omega_e} 
    &\leq C_2(\norm{\iota}_{C^1(X)}) \\
    \norm{\Rm_\sigma}_{\omega_e} 
    &\leq C_3(\norm{\iota}_{C^2(X)}) \\
    \norm{c_k(\omega_\sigma)}_{\omega_e} 
    &\leq C_4(\norm{\iota}_{C^2(X)}),
  \end{align}
  where $e\in\SL(N+1,\C)$ is the identity. By $\norm{\iota}_{C^k(X)}$ we mean
  \be
    \norm{\iota}_{C^k(X)}
    := \sup_{f\in C^k(\CP^N)\minus\set{0}} \frac{\norm{\iota^* f}_{C^k(X)} }{ \norm{f}_{C^k(\CP^N)} }
  \ee
  and the musical isomorphism $\#$ is induced by $\omega_e$.
\end{lem}

\begin{proof}
Locally,
  \be
  \omega_\sigma = \I\ddbar\log\abs{\sigma T}^2 
  = \I\sum_{i,j} h_{i\jbar} \,\dz^i\wedge\dzbar^j 
  \ee
  on an open subset $U\subset X$ so that
  \be \label{eq:Bergmanmetricincoordinates}
    h_{i\jbar}
    = \frac{\ip{\sigma(\partial_iT),\sigma(\partial_jT)}}{\abs{\sigma T}^2} - \frac{\ip{\sigma(\partial_iT),\sigma T}\ip{\sigma T,\sigma(\partial_jT)}}{\abs{\sigma T}^4},
  \ee
  where the norms and inner products in Eq.\ (\ref{eq:Bergmanmetricincoordinates}) are on $\C^{N+1}$. Recall the definition of $T:\C^n\to\C^{N+1}$ given in Section \ref{subsec:bergman}.
  
  Consider a fixed point $p\in U$. We see that each term is rational in $\sigma$ with matching degrees in the numerator and denominator. Since the norm $\abs{\sigma T}$ is nondegenerate and $T(p)\in\C^{N+1}\minus\set{0}$, each rational expression is uniformly bounded above and below. 
  Thus, at $p\in U$
  \be
     \abs{h_{i\jbar}(p)}(\sigma) \leq C_1(T(p),\bar{T}(p),\partial_iT(p),\partial_{\jbar}\bar{T}(p)),
  \ee
  which shows (\ref{eq:hest}).

  By the same token
  \begin{align}
    \partial_kh_{i\jbar}
    &= \frac{\ip{\sigma (\partial_k\partial_iT),\sigma (\partial_jT)}}{\abs{\sigma T}^2}
    - \frac{\ip{\sigma (\partial_iT),\sigma (\partial_jT)}\ip{\sigma (\partial_kT),\sigma T}}{\abs{\sigma T}^4} 
    -\frac{\ip{\sigma (\partial_k\partial_iT),\sigma T}\ip{\sigma T,\sigma (\partial_jT)}}{\abs{\sigma T}^4} \notag \\
    &\qquad 
    -\frac{\ip{\sigma (\partial_iT),\sigma T}\ip{\sigma (\partial_kT),\sigma (\partial_jT)}}{\abs{\sigma T}^4} 
    +2\frac{\ip{\sigma (\partial_kT),\sigma T}\ip{\sigma (\partial_iT),\sigma T}\ip{\sigma T,\sigma (\partial_jT)}}{\abs{\sigma T}^6},
  \end{align}
  etc., so that at $p\in U$, but suppressing this dependence in the notation,
  \begin{align}
    \abs{\partial_k h_{i\jbar}}(\sigma)
    &\leq C_2(T,\bar{T},\partial_iT,\partial_{\jbar}\bar{T},\partial_kT,\partial_k\partial_iT) \\
    \abs{\partial_{\ellbar} h_{i\jbar}}(\sigma)
    &\leq C_3(T,\bar{T},\partial_iT,\partial_{\jbar}\bar{T},\partial_{\ellbar} \bar{T},\partial_{\ellbar}\partial_{\jbar}\bar T) \\
    \abs{\partial_k\partial_{\ellbar} h_{i\jbar}}(\sigma)
    &\leq C_4(T,\bar{T},\partial_iT,\partial_{\jbar}\bar{T},\partial_{\ellbar} \bar{T},\partial_{\ellbar}\partial_{\jbar}\bar T,\partial_kT,\partial_k\partial_iT).
  \end{align}
  
To bound the inverse metric $H\inv$ just note that  the constant term of the characteristic polynomial is $\Tr (H)>0$. Applying $H\inv$ to both sides of the characteristic equation shows that $H\inv$ is a polynomial in  $H$. Thus
  \be
    \big|h^{i\jbar}\big|(\sigma)\leq C_5(T,\bar{T},\partial_iT,\partial_{\jbar}\bar{T}) ,
  \ee
  which shows (\ref{eq:hinvset}).
  
  Next, we see that
  \begin{align}
    \abs{R_{i\jbar k\ellbar}}(\sigma) 
    &\leq  \abs{\partial_k\partial_{\ellbar}h_{i\jbar}}+\abs{h^{p\qbar}}\abs{\partial_k h_{i\qbar}}\abs{\partial_{\ellbar}h_{p\jbar}} 
    \leq C_6(T,\bar{T},\partial_iT,\partial_{\jbar}\bar{T},\partial_{\ellbar} \bar{T},\partial_{\ellbar}\partial_{\jbar}\bar T,\partial_kT,\partial_k\partial_iT).
  \end{align}
  (Similarly, the Ricci and scalar curvatures are uniformly bounded with respect to $\sigma$, since contractions with $h^{i\jbar}$ are controlled.)
  Finally, we note that the Chern forms are given by polynomials of the curvature 2-form, which is uniformly bounded.
\end{proof}

\begin{lem} \label{lem:estimate}
  Assume $X\times\CP^{n-k}$ is dually nondegenerate in its Segre embedding. Then the higher K-energies have log-polynomial growth. Thus, for each $k=1,\ldots,n$, there is a holomorphic function $F_k$ on $G$ and constants $\ell_k>0$ and $C_k>0$ such that for all $\sigma \in G$,
\be \label{eq:estresult}
  (-1)^{n+1}D_k(\sigma) - \lognorm{\Delta_{X\times\CP^{n-k}}} = \log\abs{F_k(\sigma)}^2,
\ee
and
\be \label{eq:theestimate}
  F_k(\sigma) \leq C_k\cdot \dist(\sigma,W\minus G)^{\ell_k}.
\ee
\end{lem}

\begin{proof}
  We study the asymptotic behavior in $\sigma$ of 
  \begin{align} \label{eq:higherKen}
    M_k(\sigma)
    &= -(n+1)(n-k+1)V \int_0^1\int_X \dot\varphi_t \left[c_k(\omega_t)\wedge \omega_t^{n-k}-\mu_k\omega_t^n\right] \,\d t 
  \end{align}
  by considering the particular path in $\calH_\omega$ given by
  \begin{align}
    \varphi_t &= \log\frac{\abs{e^{\xi t}T}^2}{\abs{T}^2}
  \end{align}
  where $\xi\in\sl(N+1,\C)$ satisfies $e^\xi=\sigma$. 
  With this path $\omega_t:=\omega+\I\ddbar\varphi_t$ is a Bergman metric for each $t\in[0,1]$. By the previous lemma the factor in brackets in Eq. (\ref{eq:higherKen}) is uniformly bounded in $\sigma$.
  Also
  \begin{align}
    \abs{\dot\varphi_t} (\sigma)
    &= \frac{\abs{\ip{e^{\xi t} T,(\xi^*+\xi)e^{\xi t} T}}}{\abs{e^{\xi t}T}^2} 
    \leq \norm{\xi^*+\xi}_{\op} 
    \leq \log\Tr(\sigma^*\sigma),
  \end{align}
  where $\norm{\;\cdot\;}_{\op}$ is the operator norm on matrices. The last inequality follows since the eigenvalues of $\sigma$ are the exponentials of the eigenvalues of $\xi$. This establishes the estimate (\ref{eq:theestimate}); Eq.\ \ref{eq:estresult} now follows from Proposition 4.2 in \cite{paul12}.
\end{proof}

This establishes Lemma \ref{lem:paul} for the Donaldson functionals corresponding to the higher K-energies.

\section{Discriminant Degrees}
\label{sec:discrdegrees}

The purpose of this section is to compute the degree of the $X$-hyperdiscriminant of format $(n-k)$. To accomplish this we use the following result of Beltrametti, Fania, and Sommese \cite{BFS92}. 
\begin{lem}[\cite{BFS92}] \label{lem:BFS92}
If  $X^n\hookrightarrow\CP^N$ is smooth, then 
$X^\vee$ is a hypersurface if and only if $c_n(J_1(\calO_X(1))) \neq 0$, where $J_1(\calO_X(1))$ is the bundle of 1-jets of the hyperplane bundle on $\CP^N$ restricted to $X$.
In this case
\be \label{eq:BFSdegformula}
  \deg(\Delta_X) = \int_X c_n(J_1(\calO_X(1))) .
\ee
\end{lem}

In the case of the $X$-hyperdiscriminant of format $(n-k)$ 
the integral becomes
\be \label{eq:degreeintegral}
\deg\left(\Delta_X^{(n-k)}\right) = \int_{X\times \CP^{n-k}} s^* c_{2n-k}(J_1(\calO_{s(X\times\CP^{n-k})}(1))),
\ee 
where $s:X\times \CP^{n-k}\hookrightarrow \CP^\ell$, $\ell=(n+1)(n-k+1)-1$, denotes the Segre embedding. To compute this integral, our strategy will be to split Chern classes up until each factor is supported either on $X$ or on $\CP^{n-k}$. This is accomplished by the following.

\begin{lem} \label{lem:JettopChern}
We have
\be \label{eq:JettopChern}
  c_{2n-k}(J_1)
  = \sum_{i=0}^{k} (-1)^i (n-i+1)\binom{n-i}{n-k}  c_i(\omega)\wedge \omega^{n-i} \wedge \omega_{FS}^{n-k} 
\ee
where
\begin{align}
  J_1 
  &:= J_1(\calO_{s(X\times \CP^{n-k})}(1)) &
  \omega
  &:= \pr_1^*\omega=\pr_1^*c_1(\calO_X(1)) \\
  c_i(\omega)
  &:= \pr_1^*c_i(T^{1,0}_X)
  = (-1)^i\pr_1^*c_i(\Omega_X^{1,0}) &
  \omega_{FS}
  &:= \pr_2^*\omega_{FS}
  = \pr_2^* c_1(\calO_{\CP^{n-k}}).
\end{align}
\end{lem}

\begin{proof}

\noindent\textbf{Bundle Factorization Formulas.} 

\noindent\emph{Smooth} Euler Splitting:
\begin{align}
  \bigoplus^{k+1} \calO_{\CP^k}(-1) &\cong \Omega^{1,0}_{\CP^k} \oplus \calO_{\CP^k}  
\end{align}
Jet Bundle Sequence: for any holomorphic line bundle $L\to X$
\begin{align}
  \SelectTips{cm}{}
  \xymatrix{
    0 \ar[r] & \Omega^{1,0}_X \otimes L \ar[r] & J_1(L) \ar[r] & L \ar[r] & 0
  } 
\end{align}
Segre Factorization: setting $s^*\calO_{s(X\times\CP^{n-k})}(1):=s^*\calO_{\CP^\ell}(1)|_{s(X\times\CP^{n-k})}$ and $\calO_X(1):=\calO_{\CP^N}(1)|_X$
\begin{align}
  s^*\calO_{s(X\times\CP^{n-k})}(1) 
  &\cong \pr_1^*\left(\calO_X(1)\right) \otimes \pr_2^*\left(\calO_{\CP^{n-k}}(1)\right) 
\end{align}
(Holomorphic) Base Product Splitting:
\begin{align}
  s^*\Omega^{1,0}_{s(X\times\CP^{n-k})}(1) 
  &\cong \left(\pr_1^*\Omega^{1,0}_X(1) \otimes \pr_2^*\calO_{\CP^{n-k}}(1)\right) \oplus \left( \pr_1^*\calO_X(1) \otimes \pr_2^*\Omega^{1,0}_{\CP^{n-k}}(1) \right) 
\end{align}
Twisted Smooth Euler Splitting:
\begin{align}
  \bigoplus^{k+1} \pr_1^*\calO_X(1) 
  &\cong \left(\pr_1^*\calO_X(1) \otimes \pr_2^*\Omega^{1,0}_{\CP^k}(1)\right) \oplus \left(\pr_1^*\calO_X(1) \otimes \pr_2^*\calO_{\CP^k}(1) \right)  
\end{align}
\bigskip

By the jet bundle sequence for 
\begin{align}
  L &=\calO_{s(X\times \CP^{n-k})}(1):=\calO_{\CP^\ell}(1)|_{s(X\times \CP^{n-k})} \\
  J_1\! :\!\! &= J_1(L) = J_1(\calO_{s(X\times \CP^{n-k})}(1))
\end{align}
over $s(X\times\CP^{n-k})$, the total Chern class of the jet bundle is
\begin{align}
  s^*c(J_1) 
  &= s^*c\left(\Omega^{1,0}_{s(X\times\CP^{n-k})}\otimes \calO_{s(X\times\CP^{n-k})}(1)\right) \wedge 
  s^*c\left(\calO_{s(X\times\CP^{n-k})}(1)\right) \\
  &= s^*c\left(\Omega^{1,0}_{s(X\times\CP^{n-k})}(1)\right) \wedge 
  s^*c\left(\calO_{s(X\times\CP^{n-k})}(1)\right) \\
  &= c\left(s^*\Omega^{1,0}_{s(X\times\CP^{n-k})}(1)\right) \wedge 
  c\left(s^*\calO_{s(X\times\CP^{n-k})}(1)\right). 
\end{align}
Applying the holomorphic base product splitting to the first factor and the Segre factorization to the second factor, we see that
\begin{align}
  s^*c(J_1) 
  &= c\left(\left(\pr_1^*\Omega^{1,0}_X(1) \otimes \pr_2^*\calO_{\CP^{n-k}}(1)\right) \oplus \left( \pr_1^*\calO_X(1) \otimes \pr_2^*\Omega^{1,0}_{\CP^{n-k}}(1) \right)\right) \\
  & \hspace{2cm} \wedge c\left(\pr_1^*\left(\calO_X(1)\right) \otimes \pr_2^*\left(\calO_{\CP^{n-k}}(1)\right)\right) \\
  &= c\left(\pr_1^*\Omega^{1,0}_X(1) \otimes \pr_2^*\calO_{\CP^{n-k}}(1)\right) \\ 
  & \hspace{1cm} \wedge c\left( \left( \pr_1^*\calO_X(1) \otimes \pr_2^*\Omega^{1,0}_{\CP^{n-k}}(1) \right) 
  \oplus \left(\pr_1^*\left(\calO_X(1)\right) \otimes \pr_2^*\left(\calO_{\CP^{n-k}}(1)\right)\right)\right) ,
\end{align}
where we used the Whitney product formula in the second equality.
By the smooth Euler splitting this becomes
\begin{align}
  s^*c(J_1) 
  &= c\left(\pr_1^*\Omega^{1,0}_X(1) \otimes \pr_2^*\calO_{\CP^{n-k}}(1)\right) 
  \wedge c\left(\bigoplus^{n-k+1} \pr_1^*\calO_X(1)\right) \\
  &= c\left(\pr_1^*(\Omega^{1,0}_X \otimes \calO_X(1)) \otimes \pr_2^*\calO_{\CP^{n-k}}(1)\right) 
  \wedge c\left(\pr_1^*\calO_X(1)\right)^{n-k+1}  \\
  &= c\left(\pr_1^*\Omega^{1,0}_X \otimes \left(\pr_1^*\calO_X(1) \otimes \pr_2^*\calO_{\CP^{n-k}}(1)\right) \right)
  \wedge c\left(\pr_1^*\calO_X(1)\right)^{n-k+1}  .
\end{align}
To obtain $p$th Chern classes, we apply the general formula 
\be
  c_p(E\otimes L) = \sum_{i=0}^p \binom{r-i}{p-i} c_i(E)\wedge c_1(L)^{p-i},
\ee
where $E$ is a rank $r$ vector bundle, $L$ is a line bundle, and $0\leq p\leq r$ is an integer. Taking $E=\pr_1^*\Omega_X^{1,0}$ and $L=\pr_1^*\calO_X(1) \otimes \pr_2^*\calO_{\CP^{n-k}}(1)$, it follows that
\begin{align}
  s^*c(J_1)  
  &= \sum_{p=0}^n \sum_{i=0}^p \binom{n-i}{p-i} 
  c_i\!\left(\pr_1^*\Omega^{1,0}_X\right) \wedge 
  c_1\!\left(\pr_1^*\calO_X(1)\otimes \pr_2^*\calO_{\CP^{n-k}}(1)\right)^{p-i} 
  \wedge c\left(\pr_1^*\calO_X(1)\right)^{n-k+1} \\
  &= \sum_{p=0}^n \sum_{i=0}^p \binom{n-i}{p-i} 
  (-1)^i c_i(\omega) \wedge \left( \omega + \omega_{FS} \right)^{p-i} 
  \wedge (1+\omega)^{n-k+1}  \\
  &= \sum_{p=0}^n \sum_{i=0}^p \sum_{j=0}^{p-i} \sum_{q=0}^{n-k+1} \binom{n-i}{p-i} \binom{p-i}{j} \binom{n-k+1}{q} (-1)^i c_i(\omega) \wedge \omega^{p+q-j-i} \wedge \omega_{FS}^{j}.
\end{align}

We are now ready to compute $c_{2n-k}(J_1)$. When $j=n-k$ and $p+q-j=n$, it follows that 
$q=2n-k-p$. Then $p\leq n$ implies $q\geq n-k$ so that $q\in\set{n-k,n-k+1}$. This, in turn, implies that $p\in\set{n-1,n}$. Also $j\leq p-i$ implies $i\leq p-j=p-(n-k)$.
Thus
\begin{align}
  c_{2n-k}(J_1) 
  &= \sum_{p=n-1}^n \sum_{i=0}^{p-(n-k)}  \sum_{q=n-k}^{n-k+1} (-1)^i \binom{n-i}{p-i} \binom{p-i}{n-k} \binom{n-k+1}{q} 
  \, c_i(\omega) \wedge \omega^{p+q-n+k-i} \wedge \omega_{FS}^{n-k} \\
  &= \sum_{p=n-1}^n \sum_{i=0}^{p-(n-k)}  (-1)^i \binom{n-i}{p-i} \binom{p-i}{n-k} 
  \, c_i(\omega) \wedge \left[ (n-k+1)\omega^{p-i}+\omega^{p-i+1}\right] \wedge \omega_{FS}^{n-k} \\
  &= \sum_{i=0}^{k-1}  (-1)^i \left[(n-k+1) \binom{n-i}{n-k}+(n-i)\binom{n-i-1}{n-k} \right] 
  \, c_i(\omega) \wedge \omega^{n-i} \wedge \omega_{FS}^{n-k} \notag \\
  & \qquad 
+ (-1)^k (n-k+1) \, c_k(\omega) \wedge \omega^{n-k} \wedge \omega_{FS}^{n-k} .
\end{align}
A quick calculation shows that 
\be (n-k+1) \binom{n-i}{n-k}+(n-i)\binom{n-i-1}{n-k} = (n-i+1)\binom{n-i}{n-k}. 
\ee
\end{proof}


\begin{lem} \label{lem:degrees}
Let $\Delta_{X}^{(n-k)}$ denote the $X$-hyperdiscriminant of format $(n-k)$ and $\mu_k$ be as in Eq.\ (\ref{eq:muk}). Then the degree of $\Delta_X^{(n-k)}$ is given by
 \begin{align}
   \deg\left(\Delta_{X}^{(n-k)}\right) 
   &= \deg(X) \sum_{i=0}^k (-1)^i (n-i+1) \binom{n-i}{n-k}\, \mu_i \label{eq:discrdeg} .
 \end{align}
\end{lem}

\begin{proof}
Follows immediately from Eqs.\ (\ref{eq:muk}), (\ref{eq:BFSdegformula}), and (\ref{eq:JettopChern}).
\end{proof}

\section{Relations among Discriminants and Higher K-Energies} \label{sec:higherKen}

\begin{lem} \label{lem:lognorm}
Let $X\hookrightarrow\CP^N$ be a smooth, linearly normal $n$-dimensional subvariety. Assume that $\delta(X)\leq n- k$, where $\delta(X)$ is the dual defect of $X$. Then there is a continuous norm $\norm{\,\cdot\,}$ on the vector space of degree $d_k^\vee := \deg(\Delta_X^{(n-k)})$ polynomials on $(\C^{N+1}\otimes\C^{n-k+1})^\vee$ such that for all $\sigma\in\SL(N+1,\C)$, we have
\begin{align}
  \log \frac{\norm{\sigma\cdot\Delta_X^{(n-k)}}^2}{\norm{\Delta_X^{(n-k)}}^2} 
  &= \sum_{i=0}^{k} (-1)^i (n-i+1) \binom{n-i}{n-k} \int_0^1 \int_{X} \dot\varphi_t \, c_i(\omega_t) \wedge \omega_t^{n-i}\wedge \d t ,\label{eq:logdiscr}
\end{align}
where $e$ denotes the identity in $\SL(N+1,\C)$.
\end{lem}

\begin{proof}
  Combining equations (5.50) and (5.52) in \cite{paul12} we see that
\be
 D_{J_1(\calO_X(1))^\vee}(c_{n+1};H(\sigma),H(e))
 =
 (-1)\int_0^1 \int_X \dot\varphi_t c_n(J_1(\calO(1)|X)^\vee;h_t)\,\d t
\ee
one the one hand; on the other hand by the Main Lemma (p.\ 276 \emph{ibid.}), which we have extended to the higher K-energies in Lemma \ref{lem:estimate},
\be
 D_{J_1(\calO_X(1))^\vee}(c_{n+1};H(\sigma),H(e))
 =
 (-1)^{n+1} \log\frac{\norm{\sigma\cdot\Delta_X}^2}{\norm{\Delta_X}^2} .
\ee
Hence,
\begin{align}
  \log\frac{\norm{\sigma\cdot\Delta_X}^2}{\norm{\Delta_X}^2} 
  &= (-1)^n \int_0^1 \int_X \dot\varphi_t c_n(J_1(\calO(1)|_X)^\vee;h_t)\,\d t \\
  &= \int_0^1 \int_X \dot\varphi_t c_n(J_1(\calO(1)|_X);h_t)\,\d t  \label{eq:donaldsonlog}.
\end{align}
By Lemma \ref{lem:JettopChern} it follows that
\begin{align}
  \log \frac{\norm{\sigma\cdot\Delta_X^{(n-k)}}^2}{\norm{\Delta_X^{(n-k)}}^2} 
  &= \sum_{i=0}^{k} (-1)^i (n-i+1)\binom{n-i}{n-k} \int_0^1 \int_{X\times\CP^{n-k}} \dot\varphi_t \, c_i(\omega_t) \wedge \omega_t^{n-i}\wedge \omega_{FS}^{n-k} \wedge \d t \\
  &= \sum_{i=0}^{k} (-1)^i (n-i+1) \binom{n-i}{n-k} \int_0^1 \int_{X} \dot\varphi_t \, c_i(\omega_t) \wedge \omega_t^{n-i}\wedge \d t .\label{eq:logdiscr}
\end{align}
\end{proof}

\begin{Thm} \label{thm:mainformula}
Under the hypotheses of Lemma \ref{lem:lognorm}, we have
  \be
  M_k(\sigma)
  = \sum_{i=1}^k (-1)^{i+1} \binom{n-i}{n-k} \left[ \deg\left(R_X\right) \lognorm{\Delta_X^{(n-i)}} - \deg\left(\Delta_{X}^{(n-i)}\right) \lognorm{R_X} \right].
  \ee
\end{Thm}

\begin{proof}
First, note that for each $n\geq 0$ and $k\leq n$, the linear system
 \begin{align}
   Y_j &= \sum_{i=0}^j \binom{n-i}{n-j} X_i, \qquad j=0,1,\ldots,k
 \end{align}
 has the solution
 \begin{align}
   X_j &= \sum_{i=0}^j (-1)^{i+j} \binom{n-i}{n-j} Y_i, \qquad j=0,1,\ldots,k.
 \end{align}
When applied to Eqs.(\ref{eq:discrdeg}) and (\ref{eq:logdiscr}), this gives, respectively,
 \begin{align} 
   \mu_k
   &= \frac{1}{n-k+1} \sum_{i=0}^k (-1)^i \binom{n-i}{n-k} \frac{\deg\left(\Delta_{X}^{(n-i)}\right)}{ \deg(X) } \label{eq:mukbydiscrim} \\
   \int_0^1 \int_{X} \dot\varphi_t \, c_k(\omega_t) \wedge \omega_t^{n-k}\wedge \d t 
  &= \frac{1}{n-k+1} \sum_{i=0}^{k} (-1)^i \binom{n-i}{n-k}  \lognorm{\Delta_X^{(n-i)}} \label{eq:intbydiscrim}.
 \end{align}
Applying Eqs.(\ref{eq:mukbydiscrim}) and (\ref{eq:intbydiscrim}) to Eq.(\ref{eq:genKenergy}) gives the result.

\end{proof}

\begin{remark}
It is interesting that the $X$-hyperdiscriminants $\Delta_X^{(n-i)}$ of format $(n-i)$, $i=0,\ldots,k$ are collectively responsible for encoding the presence of the $k$th Chern form in $M_k$.

\end{remark}

\begin{proof}[Proof of Theorem \ref{thm:mainthm}]
The theorem now follows directly from Theorem \ref{thm:mainformula}, after gathering even and odd powers of $(-1)$. 
Explicitly, the vectors are  
\begin{align}
 v 
 &= R_X^{\sum_{j=1}^{\left\lfloor\frac{k}{2}\right\rfloor}\binom{n-2j}{n-k}d^\vee_{2j}} \otimes \displaystyle\bigotimes_{j=1}^{\left\lceil\frac{k}{2}\right\rceil} \left(\Delta_X^{(n-2j+1)} \notag\right)^{\binom{n-2j+1}{n-k}d_0^\vee} \\
 w
 &= R_X^{\sum_{j=1}^{\left\lceil\frac{k}{2}\right\rceil}\binom{n-2j+1}{n-k}d^\vee_{2j-1}} \otimes \displaystyle\bigotimes_{j=1}^{\left\lfloor\frac{k}{2}\right\rfloor} \left(\Delta_X^{(n-2j)}\right)^{\binom{n-2j}{n-k}d_0^\vee},
\end{align}
where $d^\vee_i := \deg\left(\Delta_X^{(n-i)}\right)$.
 We regard the polynomials $R_X^r$ and $\left(\Delta_X^{(n-i)}\right)^r$ as vectors in the irreducible $\SL(N+1,\C)$-modules 
\begin{align}
  R_X^r &\in \C_{rd_0^\vee}[M_{(n+1)\times (N+1)}]^{\SL(n+1,\C)} \\
  \left(\Delta_X^{(n-i)}\right)^r &\in \C_{rd_k^\vee}[M_{(n-i+1)\times (N+1)}]^{\SL(n-i+1,\C)} 
\end{align} 
for $r$ a positive integer and $i=1,2,\ldots ,k$, $\delta(X)\leq n-k$. The $\SL(N+1,\C)$-modules $V$ and $W$ are then the appropriate tensor product modules containing $v$ and $w$, respectively.
\end{proof}

\begin{remark}
From Lemmas \ref{lem:degrees} and \ref{lem:lognorm}, we have a recursion relation
\begin{align}
   M_k(\sigma) 
   &= (-1)^{k+1} \left[\deg\left(R_X\right) \lognorm{\Delta_{X}^{(n-k)}} - \deg\left(\Delta_{X}^{(n-k)}\right) \lognorm{R_X} + \sum_{i=1}^{k-1} (-1)^i \binom{n-i}{n-k} M_i(\sigma) \right].
\end{align}
\end{remark}

\begin{remark}
When $k=1$ we recover formula (1.1) in Theorem A in \cite{paul12}:
  \be
    M_1(\sigma) = \deg\left(R_X\right) \lognorm{\Delta_{X}^{(n-1)}} - \deg\left(\Delta_{X}^{(n-1)}\right) \lognorm{R_X} .
  \ee
  In this case $V$ and $W$ are \emph{irreducible}; in contrast, for $k>1$, $V$ and $W$ may no longer be irreducible.
\end{remark}

Corollary \ref{cor:asymptotics} now follows from the asympototic expansions (\cite{paul12} p.268)
\begin{align}
  \lim_{\abs{t}\too 0} \log \norm{\lambda(t) v}^2 &= w_\lambda(v) \log\abs{t}^2 + O(1) \\
  \lim_{\abs{t}\too 0} \log \norm{\lambda(t) w}^2 &= w_\lambda(w) \log\abs{t}^2 + O(1), 
\end{align}
where $v\in V$, $w\in W$, and $w_\lambda(v)$, $w_\lambda(w)$ are the weights of $\lambda$ on $v$ and $w$, respectively. 

Corollary \ref{cor:boundedness} follows from the general formula (\cite{paul13} p.18 Lemma 4.1)
\be
  \lognorm{v}-\lognorm{w} = \log \tan^2 d_g(\sigma\cdot[(v,w)],\sigma\cdot[(v,0)]),
\ee
where $d_g$ denotes the distance in the Fubini-Study metric on $\P(V\oplus W)$, and the numerical criterion established in \cite{paul12a}.

\bibliographystyle{alpha}

\bibliography{/Users/admin/Dropbox/LaTeX_Header_FIles/KahlerGeometryBibItems}

\end{document}